\documentclass[11pt]{article}
\usepackage{amsmath, amssymb, amscd, amsthm, amsfonts, subfig}
\usepackage{graphicx}
\usepackage{hyperref}
\usepackage{verbatim}
\usepackage{float}

\usepackage[font=small,labelfont=bf]{caption}
\usepackage{titlesec}
\expandafter\def\expandafter\normalsize\expandafter{%
    \normalsize%
    \setlength\abovedisplayskip{5pt}%
    \setlength\belowdisplayskip{5pt}%
    \setlength\abovedisplayshortskip{-3pt}%
    \setlength\belowdisplayshortskip{3pt}%
}
\usepackage[scr=boondox, % heavily sloped
            cal=esstix] % slightly sloped
           {mathalpha}
\newtheorem{theorem}{Theorem}[section] % reset theorem counter every section

\newtheorem{corollary}{Corollary}[theorem] % reset corollary counter every theorem
\newtheorem{lemma}[theorem]{Lemma} % lemmas, propositions, claims, and conjectures all share numbering with theorems
\newtheorem{proposition}[theorem]{Proposition}

\newtheorem{conjecture}[theorem]{Conjecture}

\theoremstyle{definition}
\newtheorem{definition}{Definition}% definitions, examples, etc. have their own numbers
\newtheorem{example}{Example}

\theoremstyle{remark}
\newtheorem*{remark}{Remark} % remarks have no counter. 

\usepackage{color,colortbl}

\DeclareMathOperator{\zigzag}{zigzag}
\DeclareMathOperator{\ballot}{ballot}
\DeclareMathOperator{\degree}{deg}
\title{On Chip-Firing on Undirected Binary Trees}
\author{Ryota Inagaki \and Tanya Khovanova \and Austin Luo}
\date{}

\begin{document}

\maketitle

\begin{abstract}
Chip-firing is a combinatorial game played on an undirected graph in which we place chips on vertices. We study chip-firing on an infinite binary tree in which we add a self-loop to the root to ensure each vertex has degree 3. A vertex can fire if the number of chips placed on it is at least its degree. In our case, a vertex can fire if it has at least 3 chips, and it fires by dispersing $1$ chip to each neighbor. Motivated by a 2023 paper by Musiker and Nguyen on this setting of chip-firing, we give an upper bound for the number of stable configurations when we place $2^\ell - 1$ labeled chips at the root. When starting with $N$ chips at the root where $N$ is a positive integer, we determine the number of times each vertex fires when $N$ is not necessarily of the form $2^\ell - 1$. We also calculate the total number of fires in this case.
\end{abstract}

%\graphicspath{{D:/LATEX/Reports@IIT/figures/}}
\section{Introduction}

The game of chip-firing depicts a dynamical system and is an important part in the field of structural combinatorics. Chip-firing originates from problems such as the Abelian sandpile \cite{dhar1999abelian}, which states that when a stack of sand grains exceeds a certain height, the stack will disperse grains evenly to its neighbors. Eventually, the sandpile may achieve a stable configuration, which is when every stack of sand cannot reach the threshold to disperse. This idea of self-organizing criticality combines a multitude of complex processes into a simpler process. Chip-firing as a combinatorial game began from the works such as those of Spencer \cite{MR856644} and Anderson, Lovasz, Shor, Spencer, Tardos and Winograd \cite{zbMATH04135751}. There are many variants of the chip-firing game \cite{MR3311336, MR3504984, MR4486679} which allow for the discovery of unique properties. For example, in \cite{MR3311336, MR3504984}, certain classes of stable configurations can be described as a critical group. When the chips are distinguishable, this unique property fails, prompting a new area of study.

\subsection{Unlabeled chip-firing on undirected graphs}
\label{sec:unlabeledchipfiring}

 Unlabeled chip-firing occurs when indistinguishable chips are placed on vertices in a graph. If a vertex has enough chips to transfer one chip to each neighbor, then that vertex can fire. In other words, if there are at least $\degree(v)$ chips on a vertex $v$, it can fire. When a vertex fires, it sends one chip to each neighbor and thus loses $\degree(v)$ chips. Once all vertices can no longer fire, we consider this to be a \textit{stable configuration} (see Section~\ref{sec:definitions} for the full definition).

\begin{example}
    In Figure~\ref{fig:exampleunlabel}, the unlabeled chip-firing process is shown when we start with $3$ chips on a vertex.
\end{example}

\begin{figure}[H]
\centering
    \subfloat[\centering Initial configuration with $3$ chips]{{\includegraphics[width=0.35\linewidth]{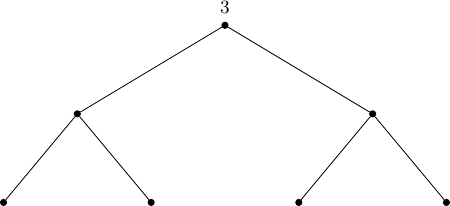} }}%
    \qquad
    \subfloat[\centering Stable configuration after firing once]{{\includegraphics[width=0.35\linewidth]{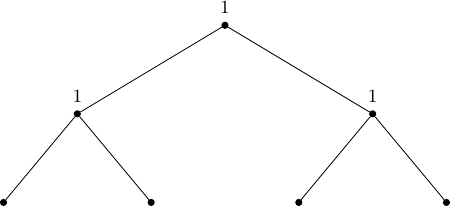} }}%
    \caption{Example of unlabeled chip-firing}%
    \label{fig:exampleunlabel}
\end{figure}

Let us define a \textit{configuration} $c$ as a distribution of chips over the vertices of a graph which is represented as a vector $\vec{c}$ in $\mathbb{N}^{\infty}$ (the set of infinite sequences indexed by the nonnegative integers whose entries are real numbers) where the $k$th entry in $\vec{c}$ corresponds with the number of chips on vertex $v_k$. The properties of unlabeled chip-firing have been studied extensively by Klivans \cite{klivans2018mathematics}. Important properties to highlight are \textit{Local Confluence} and \textit{Global Confluence}. 

\begin{theorem}[Theorem 2.2.2 in \cite{klivans2018mathematics}]
\label{thm:confluence}
Let $c$ be a configuration of unlabeled chips on an undirected graph.
\begin{enumerate}
    \item \textbf{Local Confluence.} If we obtain $c_1$ and $c_2$ after one firing from configuration $c$, then there is a common configuration $d$ that can be achieved after firing once from $c_1$ or $c_2$.
    \item \textbf{Global Confluence.} Let $c_s$ be the stable configuration. If $c_s$ can be reached from $c$ via a finite number of firings, then $c_s$ is a unique stable configuration. 
\end{enumerate}
\end{theorem}

This tells us that the order of firings, when starting with indistinguishable chips, does not matter as it will always result in the same stable configuration. As a Corollary of the above theorem,  regardless of the sequence of firings producing a stable configuration, each vertex fires the same number of times during the process of stabilization (see \cite{klivans2018mathematics}).

\subsection{Labeled chip-firing}
\label{sec:labeledchipfirng}

Labeled chip-firing is a variant of chip-firing where the chips are distinguishable. We denote this by assigning each chip a number from the set of $\{1,2,\dots, N\}$ where there are $N$ chips in total. A vertex $v$ can fire if it has at least $\degree(v)$ chips. When a vertex fires, we choose any $\degree(v)$ labeled chips and disperse them, one chip for each neighbor. The chip that each neighbor receives may depend on the label of the chip. Labeled chip-firing was originally studied in the context of one-dimensional lattices \cite{MR3691530}.

In this paper, we study labeled chip-firing in the context of infinite undirected binary trees. We use the setup in~\cite{musiker2023labeledchipfiringbinarytrees} in which a self-loop is added to the root. The self-loop allows each vertex to have $3$ neighbors, making the game's mechanics easier to study since a vertex $v$ can fire in this setup if it has at least $\degree(v) = 3$ chips. An example of the setup is shown in Figure~\ref{fig:binarytreesetup}.

\begin{figure}[H]
    \centering
    \includegraphics[width=0.55\linewidth]{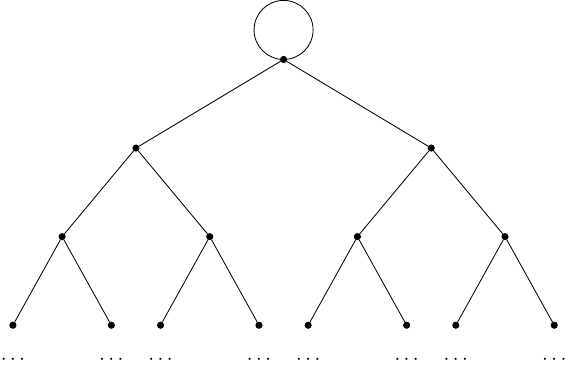}
    \caption{Binary tree setup with self-loop}
    \label{fig:binarytreesetup}
\end{figure} 

Initially, we place a positive integer $N$ labeled chips at the root. When a vertex fires, we arbitrarily select $3$ labeled chips on that vertex and order them numerically from least to greatest via their labels. We disperse the chips such that the smaller chip is sent to the left child, the larger to the right, and the middle to the parent or to itself in the case of the root. 

In labeled chip-firing, an important property is that Theorem~\ref{thm:confluence} does not hold. This means that we can achieve different stable configurations depending on the triples of chips we arbitrarily select to fire. More precisely, two stable configurations would always have the same number of chips at each vertex, but the labels might differ.

\begin{example}
    Consider starting with five labeled chips at the root. One possible strategy is firing a triple $(2,3,4)$ followed by $(1,3,5)$. Another possible strategy is firing a triple $(1,2,3)$ followed by $(2,4,5)$. Figure~\ref{fig:noconfluence} illustrates the two different stable configurations, breaking confluence. 
\end{example}
\begin{figure}[H]
\centering
    \subfloat[\centering A $(2,3,4), (1,3,5)$ firing pattern]{{\includegraphics[width=0.35\linewidth]{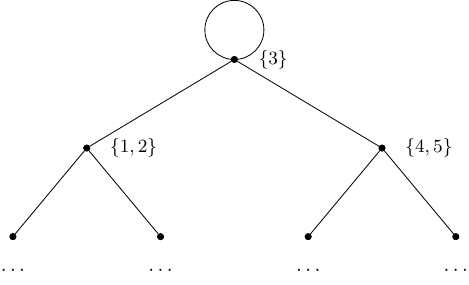} }}%
    \qquad
    \subfloat[\centering A $(1,2,3), (2,4,5)$ firing pattern]{{\includegraphics[width=0.35\linewidth]{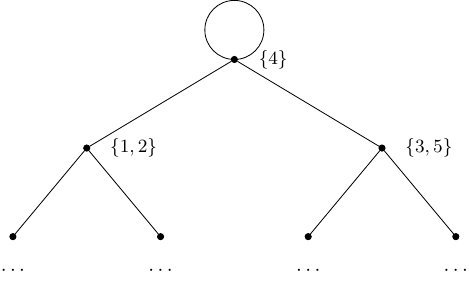} }}%
    \caption{Example of confluence not holding}%
    \label{fig:noconfluence}
\end{figure}

\subsection{Motivation and objectives}
\label{sec:motivations}

 Musiker and Nguyen \cite{musiker2023labeledchipfiringbinarytrees} have studied both unlabeled and labeled chip-firing on infinite binary trees intensively when starting with $2^{\ell} - 1$ chips at the root for some $\ell \in \mathbb{N}^+$. They provided proofs on the stable configuration in the unlabeled case and the number of times each vertex fires in the process of reaching the stable configuration. In discussing their results, they touch upon determining the number of stable configurations in the labeled game. This motivates us to further explore the properties of the stable configuration and how we achieve certain stable configurations in the variants of chip-firing. Motivated by their work, we ask the following questions in the setting of chip-firing on the infinite undirected binary tree with a self-loop at the root:
\begin{itemize}
        
       \item When starting with $2^{\ell} - 1$ labeled chips at the root, can we bound the number of stable configurations?
        \label{obj2}
        
       \item When starting with $N$ unlabeled chips at the root for any $N \in \mathbb{N}^+$, how many times does each vertex fire upon reaching the stable configuration? What is the total number of fires?
        \label{obj3}
\end{itemize}

\subsection{Roadmap}

\label{sec:roadmap}

In Section~\ref{sec:prelim}, we define important notation and terms and provide an overview of the prior results used in this paper. In Section~\ref{sec:bounds}, we discuss the difference between an initial tree and a subtree in a stable configuration. We use the zigzag method to upper-bound the number of stable configurations when starting with $2^\ell - 1$ labeled chips at the root for any $\ell \in \mathbb{N}^+$. We also use a conjectural property of each subtree in the stable configuration to provide a stronger upper bound. In Section~\ref{sec:countingvertexfires}, we present and prove a formula for the number of times each vertex fires beginning with any positive number of chips (not just those of form $2^{\ell}-1$) at the root. We also compute the total number of fires.

\section{Preliminaries}
\label{sec:prelim}

\subsection{Definitions}\label{sec:definitions}

In this paper, we consider an infinite undirected binary tree. 

In a \textit{rooted tree}, we denote one distinguished vertex as the \textit{root} vertex $r$. Every vertex in the tree, excluding the root, has exactly one parent vertex. A vertex $v$ has \textit{parent} $v_p$ if on the path taken from vertex $v$ to the root, vertex $v_p$ is the first vertex traversed. If a vertex $v$ has parent $v_p$, then vertex $v$ is the \textit{child} of $v_p$. 

We define an \textit{infinite undirected binary tree} as an infinite undirected rooted tree with a root vertex $r$ where each vertex has two children, called the \textit{left child} and \textit{right child}, and no two vertices have children in common. We also add a self-loop at the root to guarantee that each vertex $v$ has degree 3: $\degree(v) = 3$. Thus, a vertex $v$ can \textit{fire} if it has at least $3$ chips. 

 We define the \textit{initial state} of chip-firing as placing $N$ chips on the root where, in the case of labeled chip-firing, they are labeled $1,2, \dots, N$. The \textit{stable configuration} is a distribution of chips over the vertices of a graph such that no vertex can fire. 

We define a vertex $v_j$ to be on \textit{layer} $i+1$ if the path of vertices traveled from the root to $v_j$ traverses $i$ vertices. Thus, the root $r$ is on layer $1$. 

The following is the labeling procedure for vertices. We label the root as vertex $1$, denoted $v_1$. For any vertex $i$, we label the left child as $2i$ and the right child as $2i+1$. Figure~\ref{fig:labelingexs} represents labeling for the first $4$ layers in the binary tree.

\begin{figure}[H]
    \centering
    \includegraphics[width=0.65\linewidth]{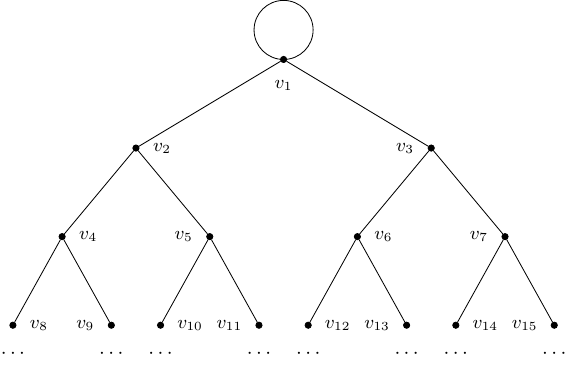}
    \caption{Labeling of vertices}
    \label{fig:labelingexs}
\end{figure}

Let $v_n$ be the root of a subtree. We denote the \textit{straight left descendant} of a vertex $v_n$ as any vertex $v_j$ where $j > i$ such that if we take the path of vertices from $v_j$ to $v_n$, each vertex on the path traversed is the left child of their parent and the \textit{straight right descendant} is defined similarly. If the straight left descendant of a vertex $v_n$ is on the last layer with chips in the stable configuration, it is called the \textit{bottom straight left descendant}, and the \textit{bottom straight right descendant} is defined similarly. 

We denote the \textit{straight ancestor} of a vertex $t$ as any vertex $v$ if vertex $t$ is a straight left or right descendent of vertex $v$. A vertex $v$ is a \textit{top straight ancestor} of vertex $t$ if vertex $v$ is the left child of its parent and vertex $t$ is a straight right descendent of vertex $v$ or vice versa. In the case of the root, it is considered the top straight ancestor of the left and right descendants.

\subsection{Prior results}

In an infinite undirected binary tree, Musiker and Nguyen \cite{musiker2023labeledchipfiringbinarytrees} show the following for stable configurations that result from initially placing $2^{\ell}-1$ labeled chips $1, 2, 3, \dots, 2^{\ell}-1$ at the root.

\begin{proposition}[Theorem 1.2 in \cite{musiker2023labeledchipfiringbinarytrees}]
\label{prop:MNfinalconfig}
        The stable configuration always has one chip at each vertex of the first $\ell$ layers. Moreover, the bottom straight left and right descendants of a vertex respectively contain the smallest and largest chips of its subtree.
\end{proposition}

\begin{proposition}[Proposition 4.4 in \cite{musiker2023labeledchipfiringbinarytrees}]
    For any vertex whose top straight ancestor is its parent, if the vertex is a left child, then the chip it contains is smaller than both those of the parent and the parent's right child and vice versa.
\label{prop:MNbottomsstraightleft}
\end{proposition}

\begin{proposition}[Proposition 4.5 in \cite{musiker2023labeledchipfiringbinarytrees}]\label{prop:PenultimateSmallLarge}
    If vertex $v$ is a parent of vertices in the $n$th layer and if $v'$ such that $v$ is a straight left (right) descendant of a vertex $v'$, then the chip at $v'$ is the smallest (largest) chip in the subtree rooted at $v'$ excluding chips on the bottom layer.
\end{proposition}

Proposition~\ref{prop:MNfinalconfig} tells us what the unique stable configuration is and the special properties of the bottom straight left and right descendants of a vertex. In Section~\ref{sec:bounds}, we use Proposition~\ref{prop:MNfinalconfig} to show an upper bound on the number of stable configurations. Proposition~\ref{prop:MNbottomsstraightleft} and Proposition~\ref{prop:PenultimateSmallLarge} provide information on where certain labeled chips can appear in the stable configuration. In Section~\ref{sec:bounds}, we use Proposition~\ref{prop:MNbottomsstraightleft} and Proposition~\ref{prop:PenultimateSmallLarge} to show an upper bound on the number of stable configurations.

\begin{proposition}[Corollary 3.4 in \cite{musiker2023labeledchipfiringbinarytrees}]
\label{prop:numchipseachlayer}
If we start with $N$ chips at the root, where $2^n-1 \leq N \leq 2^{n+1} -2$, then for $0 \leq i \leq n-1$, the resulting stable configuration has $a_i+1$ chips on each vertex on layer $i+1$, where $a_na_{n-1}\ldots a_2a_1a_0$ is the binary expansion of the number $N+1$.
\end{proposition}

Proposition~\ref{prop:numchipseachlayer} generalizes the final appearance of the stable configuration when starting with any positive integer number of chips at the root. 
In Section~\ref{sec:countingvertexfires}, we use Proposition~\ref{prop:numchipseachlayer} to generalize the number of times each vertex fires when starting with any positive integer number of chips at the root. 

\begin{proposition} [Proposition 3.6 in \cite{musiker2023labeledchipfiringbinarytrees}]
\label{prop:level-fire}
    If we start with $2^\ell-1$ chips at the root, during the firing process, for every $0\leq i< \ell$, every vertex on layer $\ell-i$ fires $2^i - i -1$ times.
\end{proposition}

Proposition~\ref{prop:level-fire} gives the number of times each vertex fires during the firing process when starting with $2^\ell - 1$ chips at the root. In Section~\ref{sec:countingvertexfires}, we use Proposition~\ref{prop:level-fire} to confirm a special case of our generalization on the number of times each vertex fires during the firing process for any positive integer $N$ chips at the root.

\section{Bounding the number of stable configurations}
\label{sec:bounds}

\subsection{General discussion}

As mentioned previously, in \cite{musiker2023labeledchipfiringbinarytrees}, a question Musiker and Nguyen leave unanswered is bounding or determining the exact number of stable configurations of labeled chips in an infinite binary tree with a self-loop at the root when starting with $2^{\ell}-1$ labeled chips (labeled $1, 2, 3, \dots, 2^{\ell}-1$) at the root. Previously, no non-trivial upper bound was produced, motivating us to find one. 

We start by observing that the two smallest and largest chips end up in the fixed positions in all stable configurations of chips in the whole tree.

\begin{lemma}\label{lem:SecondLargestAndSmallest}
    Suppose we start with $2^{\ell}-1$ labeled chips $1, 2, \dots, 2^{\ell}-1$ at the root of an infinite binary rooted tree with self-loop. Then, in the stable configuration of the tree, chip $2$ is at the parent of the vertex containing chip $1$, and chip $2^{\ell}-2$ is at the parent of the vertex containing chip $2^{\ell}-1.$
\end{lemma}

\begin{proof}
    First, we observe that if a vertex $v$ disperses chip $2$ during firing, chip $2$ is never transferred to the right child of $v$, as chip $2$ is never the largest out of three chips. Thus, chip $2$ always remains among the straight left descendants of the root of the tree. Note that chip $2$ cannot be in the straight left descendant of the root vertex at the $\ell$th layer of the tree: this is since chip $1$ is at that vertex. Because of this, chip $2$ is the smallest chip in the tree when excluding the chips in layer $\ell$. By Proposition \ref{prop:PenultimateSmallLarge}, chip $2$ has to be in a vertex at layer $\ell-1$, implying that chip 2 is at the parent of the vertex in layer $\ell$ containing chip $1$.

    Using a symmetric argument, we see that chip $2^{\ell}-2$ has to be at the parent of the vertex containing chip $2^{\ell}-1$.
\end{proof}

It is tempting to use recursion to bound the number of stable configurations when we start with $2^\ell-1$ labeled chips at the root of the infinite binary tree with a self-loop at the root. However, it is important to be aware that the subtrees do not behave the same way as the original rooted tree, as the following example shows.

\begin{example}
\label{ex:subtreenotatree}
From Patrick Liscio's list of all possible stable configurations when starting with 15 chips \cite{Liscio} (referenced in \cite{musiker2023labeledchipfiringbinarytrees}), we pick the one depicted in Figure~\ref{fig:subtreecounterexample}. In the left branch, the second largest chip, valued at 7, is not the parent of the largest chip, valued at 11.
\begin{figure}[H]
    \centering
    \includegraphics[width=0.45\linewidth]{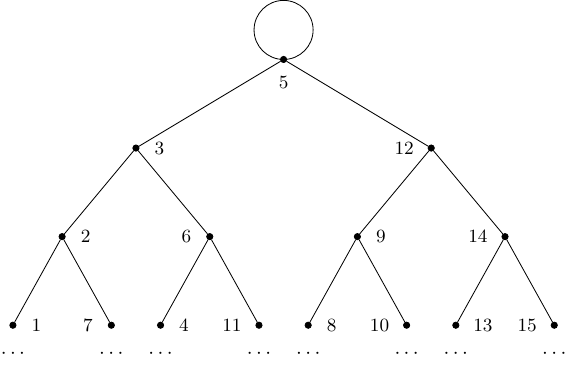}
    \caption{Counterexample showing that subtrees are not order-isomorphic to the whole tree}
    \label{fig:subtreecounterexample}
\end{figure}
\end{example}

In Example~\ref{ex:subtreenotatree}, the second largest value in the subtree does not have a guaranteed place. Thus, in a subtree, we only know the placements of the smallest and the largest chips.

We denote the number of stable configurations when we start with $2^\ell -1$ chips at the root as $Z_\ell$. These numbers for $\ell = 1,2,3$ can be found via brute force and are shown in Section $7$ of~\cite{musiker2023labeledchipfiringbinarytrees}. The number of stable configurations for when $\ell = 4$ was computed by Patrick Liscio \cite{Liscio} and is found in the appendix of \cite{musiker2023labeledchipfiringbinarytrees}. We have $Z_1 = Z_2 = 1$ and $Z_3 = 6$, and $Z_4=36220$.

Let us denote by $T_i$ the number of possible orderings of chips on a subtree that has $2^{i}-1$ chips. Alternatively, we can define $T_i$ to be the number of possible orderings of chips on a subtree so that each vertex in its first $i$ layers has a chip and no chips are in vertices below layer $i.$

\begin{example}
\label{ex:T1T2}
    If $i=1$, we have only one chip in the subtree and $T_1 = 1$. If $i=2$, the tree has two layers with a chip on each vertex, with the bottom layer consisting of two children of the same vertex $v$. By Proposition~\ref{prop:MNbottomsstraightleft}, the right and left child of $v$ respectively contain the smallest and largest chip in the subtree rooted at $v$. Thus $T_2=1$.
\end{example}

Now we estimate $T_3$. We assume that chip $c_i$ is at the vertex $v_i$. We know that $c_5 > c_2 < c_1$ and $c_6 < c_3 > c_1$, Also, $c_4 = 1$ and $c_7=7$. It follows that $c_1$ can be one of 3, 4, and 5. Suppose $c_1 = 3$, then $c_2 = 2$, and choosing $c_5$ uniquely defines the rest. Thus, there are 3 such cases. By symmetry, there are 3 cases when $c_1=5$. Suppose $c_1 = 4$. If $c_2 = 2$, then we have 3 cases depending on what $c_5$ is. If $c_2 = 3$, then $c_6 = 2$, and we have two cases. The total number of cases is 11. Thus, $T_3 \leq 11$. We checked the orders of subtrees with 7 chips in the dataset of all possible stable configurations when starting with 15 chips \cite{Liscio} (see also \cite{GithubCode}). There were 10 possible orders. The only order that is missing is the only subtree shown in Figure~\ref{fig:missingordersubtree}, where the second smallest chip is not the parent of the smallest chip, and the second largest chip is not the parent of the largest chip.
\begin{figure}[H]
    \centering
    \includegraphics[width=0.4\linewidth]{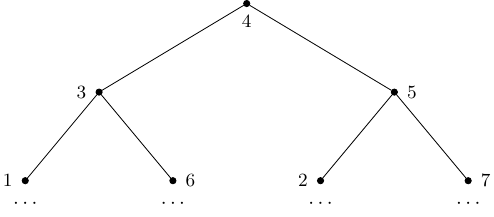}
    \caption{The order that cannot appear on a subtree}
    \label{fig:missingordersubtree}
\end{figure}

\begin{lemma}
    The subtree ordered as in Figure~\ref{fig:missingordersubtree} never appears in a stable configuration.
\end{lemma}

\begin{proof}
    Suppose the subtree above appears as a left child of its parent in a stable configuration. The smallest chip $c_1$ has to be at vertex $v_4$. We show that the second smallest chip $c_2$ is at the parent vertex of the smallest chip, vertex $v_2$. The second smallest chip cannot be at $v_5$, as otherwise, the chip at $v_2$ is smaller than $c_2$. If it is not at $v_2$, then it is in the right branch of the root of the subtree.
    
    Let $(c_1, c', c_2)$  (where $c_1 < c' < c_2$)  be the chips fired when $c_2$ was fired from $v_1$ for the last time. The end game of this chip-firing process is described in detail in \cite{musiker2023labeledchipfiringbinarytrees}. We know that after $v_2$ fires, it has zero chips left, and the chip $c'$ will be at the vertex $v_1$ for its last fire. During this fire, either $c'$ or a smaller-valued chip is fired to the left. As a result, the vertex $v_2$ contains the second smallest value of the subtree.

    By symmetry, the subtree that is the right child cannot have the second largest chip in its left branch.
\end{proof}

\begin{corollary}\label{cor:KeyCorT3}
    We have $T_3 = 10$.
\end{corollary}

The following theorem describes the connection between $T_{\ell}$ and $Z_{\ell}$.
\begin{theorem}\label{thm:TZconnection}
    We have the equality $Z_1 = T_1 = Z_2 =T_2 = 1$ and inequality $Z_{\ell} < T_{\ell}$ for $\ell >2$.
\end{theorem}

\begin{proof}
    The equality follows from Example~\ref{ex:T1T2} and known values for $Z_1$ and $Z_2$.

    Suppose every fire at the root is of the form $(i,2^{\ell-1},i+2^{\ell-1})$. In addition, assume that the left and right subtrees of our original tree fire in parallel with respect to the same order. Then, the behavior of the left subtree is isomorphic to the behavior of a tree with $2^{\ell-1}-1$ chips. Thus, for any distribution of chips of a tree, a chip-firing game exists with the same distribution of chips on a subtree.

    On the other hand, as shown in Example~\ref{ex:subtreenotatree}, for $\ell=3$, there exists an ordering of a subtree that does not match any ordering of a tree. By the argument in the previous paragraph, we can reproduce similar orderings in larger trees. It follows that for any $\ell>2$, there exists an ordering of a subtree that does not match any ordering of a tree.
\end{proof}

We now state the naive bounds for $Z_\ell$ and $T_\ell$. 

\begin{proposition}[Naive Bounds]
\label{prop:naivebounds}
    For $\ell > 2$, we have
    \[Z_\ell \leq (2^\ell-5)! \quad \textrm{ and } \quad T_\ell \leq (2^\ell-4)!.\]
\end{proposition}

\begin{proof}
  For $Z_\ell$, we are guaranteed the locations of four chips: the two smallest and two largest. Thus, if we permute the rest, the bound is $Z_\ell \leq (2^\ell-5)!$. Similarly, for the subtree, we know the locations of the smallest and the largest chips by Proposition \ref{prop:MNbottomsstraightleft}. If the root of the subtree is the left child of its parent, we know that the second smallest chip in $T_{\ell}$ must be at the parent of the vertex containing the smallest chip. Therefore, we know the locations of three chips in such a subtree. A similar argument works for a subtree whose root is the right child of its parent. Thus, $T_\ell \leq (2^\ell-4)!$.
\end{proof}

Below, we consider two more advanced methods for bounding $T_\ell$, and consequently $Z_\ell$.

\subsection{The zigzag method}

Here, we show an upper bound for the number of stable configurations using a zigzag method.

\begin{definition}
    We define a \textit{zigzag} of length $\ell$ to be a path of $\ell$ vertices beginning from vertex $v_{a_1}$ and going down, such that we alternate traveling left and right and stop at layer $\ell$. In addition, if $v_{a_1}$ is the left child of its parent, then the first move of the zigzag is to the right. If $v_{a_1}$ is the right child of its parent, then the first move of the zigzag is to the left. If $v_{a_1}$ is the root, the first move can go either right or left.
\end{definition}

More formally, suppose vertices $v_{a_1}, v_{a_2}, \dots, v_{a_n}$ form a zigzag. If $v_{a_1}$ is the left child, then for each $2 \leq i \leq n$, we have $a_i = 2a_{i-1}$ if $a_{i-1}$ is odd and $a_i = 2a_{i-1} + 1$ if $a_i$ is even. If $v_{a_1}$ is the right child, then for each $2 \leq i \leq n$, we have $a_i = 2a_{i-1}+1$ if $a_{i-1}$ is odd and $a_i = 2a_{i-1}$ if $a_i$ is even. Note that we can form two zigzags starting at the root and one zigzag starting at any other vertex.

Using Proposition~\ref{prop:MNbottomsstraightleft}, the chips on a zigzag that starts on the left child or at the root with a move to the right satisfy alternating inequalities, $c_1 < c_2 > c_3 < \dots > c_n$ or $c_1 < c_2 > c3 < \dots < c_n$ depending on the parity of $n$. Similarly, the chips on the zigzag starting from the right child or starting left at the root of the tree satisfy $c_1 > c_2 < c_3 > \dots > c_n$ or $c_1 > c_2 < c_3 > \dots <c_n$. In Figure~\ref{fig:zigzag}, three different zigzags are shown in bold: one of the zigzags starting left at the root of length $5$, a zigzag of length $4$ starting at a right child $v_3$, and a zigzag of length $3$ starting at a left child $v_4$. 

\begin{figure}[H]
    \centering
    \includegraphics[width=0.7\linewidth]{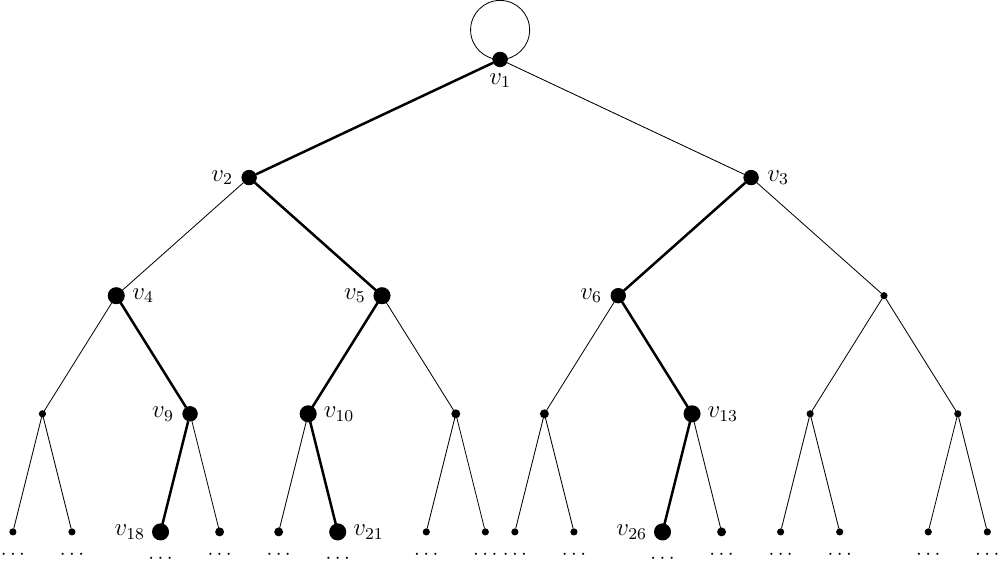}
    \caption{Zigzags of lengths $3,5,4$}
    \label{fig:zigzag}
\end{figure}

\begin{remark}
    The longest zigzag in Figure~\ref{fig:zigzag} has a particularly nice description. The indices on this zigzag form sequence 1, 2, 5, 10, 21, which are the first five terms of sequence A000975 in OEIS. This is the sequence of numbers whose binary representation has alternating digits.
\end{remark}

We also show an example of how a zigzag starting at the root divides the rest into subtrees.
\begin{example}
    Suppose we start with a tree of $4$ layers. In Figure~\ref{fig:splitting}, once we take the zigzag of length $4$ out, the vertices that are not on the zigzag can be partitioned into a subtree with $3$ layers, a subtree with $2$ layers, and a subtree with $1$ layer. 
\end{example}
\begin{figure}[H]
    \centering
    \includegraphics[width=0.5\linewidth]{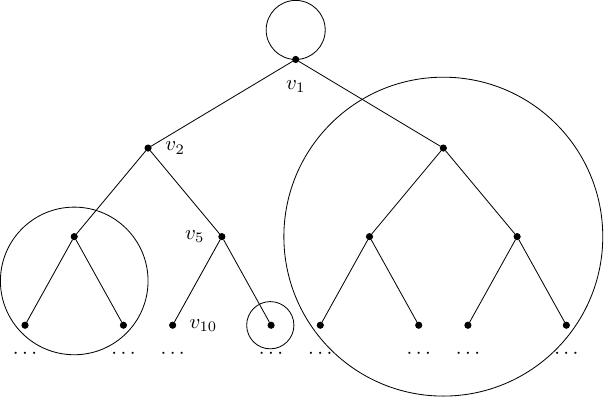}
    \caption{An example of partitioning the binary tree with $2^4 - 1$ chips}
    \label{fig:splitting}
\end{figure}

We use $A_n$ to represent the Euler zigzag numbers (A000111 in OEIS \cite{oeis}), where $A_n$ counts the number of permutations of a sequence $a_1, a_2 , \dots , a_n$ such that $a_i < a_{i+1}$ if $i$ is odd and $a_i > a_{i+1}$ if $i$ is even for all $i \in \{1,2, \dots, n\}$ or vice versa. Note that by this definition, $A_{\ell}  < \ell!$ for $\ell > 1$.

Now, we show an upper bound for $T_\ell$ and $Z_\ell$.

\begin{lemma}
    For $\ell \geq 4$, the number $T_\ell$ of stable configurations  for a subtree with $2^\ell - 1$ labeled chips and the number $Z_\ell$ of stable configurations of a tree with $2^{\ell}-1$ labeled chips are bounded as follows 
    \[T_\ell \leq 10A_{\ell} \cdot \binom{2^{\ell}- 3}{\ell} \binom{2^{\ell}-3-\ell }{ 2^{\ell-1}-2, 2^{\ell-2}-2, 2^{\ell-3}-1, 2^{\ell-4}-1, \dots, 1}T_{{\ell}-1}T_{{\ell}-2}\cdots T_{4}\]
    and 
    \[Z_\ell \leq 10A_{\ell} \cdot \binom{2^{\ell}- 5}{\ell} \binom{2^{\ell}-5-\ell }{ 2^{\ell-1}-3, 2^{\ell-2}-3, 2^{\ell-3}-1, 2^{\ell-4}-1, \dots, 1}T_{{\ell}-1}T_{{\ell}-2}\cdots T_{4}.\]
    \label{lem:recursiveupper}
\end{lemma}

\begin{proof}
    Suppose we have a subtree with $2^{\ell} - 1$ labeled chips at the root; the zigzag has length $\ell$. We first count the number of possible sequences of chips that are on the zigzag’s vertices. We know the positions of the smallest and largest chips in the subtree. For $\ell \geq 4$, the zigzag does not overlap with these positions. Thus, we can choose the set of chips in the zigzag in $\binom{2^{\ell}-3}{\ell}$ ways. The number of ways of ordering the chips in the zigzag is $A_{\ell}$. Therefore, there are at most $A_{\ell}\binom{2^{\ell}-5 }{\ell}$ possible subsequences of chips in the zigzag.

    We observe that the zigzag splits the initial subtree into subtrees with $\ell - 1, \ell -2, \dots, 1$ layers with chips so that each subtree with $\ell-i$ layers has $2^{\ell - i}-1$ chips. Moreover, the smallest chip will end up in the leftmost vertex of the $\ell$th layer of the subtree, and the largest chip will end up in the rightmost vertex of the $\ell$th layer of the subtree. Thus, we need to pick an addition $2^{\ell-1}-2$ chips for the former and $2^{\ell-1}-2$ chips for the latter subtrees. Overall, there are at most $\binom{2^{\ell}-3-\ell}{2^{\ell-1}-2, 2^{\ell-2}-2, 2^{\ell-3}-1, 2^{\ell-4}-1, \dots 1}$ ways to distribute the remaining chips into the subtrees.

    In addition, for each $i$, the number of orderings of chips in a subtree with $i$ layers is bounded by $T_{i}$. Thus the number of configurations of chips in all subtrees is bounded as $T_{\ell-1}T_{\ell-2}\cdots T_2 T_1 \leq 10T_{\ell-1}T_{\ell-2}T_{\ell-3}\cdots T_4$, with the last inequality holding because $T_1 = 1$, $T_2 = 1$, and $T_3 = 10$. Therefore, given the sequence that appears in the zigzag, we obtain that there are at most
    \[10\binom{2^{\ell}-3-\ell}{2^{\ell-1}-2, 2^{\ell-2}-2,2^{\ell-3}-1, 2^{\ell-4}-1, \dots, 1} T_{\ell-1}T_{\ell-2}T_{\ell-3}\cdots T_4\]
    ways that the remaining chips can be arranged in a stable configuration over the vertices that are not in the zigzag. Multiplying by $A_{\ell}\binom{2^{\ell}-3 }{ \ell}$, which is the number of possibilities for the set of chips that end up in the zigzag, we obtain that there are at most
    \[10 A_{\ell}\binom{2^{\ell}- 3}{\ell} \binom{2^{\ell}-3-\ell }{2^{\ell-1}-2, 2^{\ell-2}-2, 2^{\ell-3}-1, 2^{\ell-4}-1, \dots, 1}T_{{\ell}-1}T_{{\ell}-2}\cdots T_{4}\] 
    stable configurations of labeled chips in the subtree $T_\ell$.

    We now consider the whole tree. We know the positions of the $1$st, $2$nd, $2^{\ell}-2$nd, $2^{\ell}-1$st smallest chips of tree by Proposition~\ref{prop:MNfinalconfig} and Lemma~\ref{lem:SecondLargestAndSmallest}; and for $\ell \geq 4$, the zigzag does not overlap with these positions. Therefore, there are at most $\binom{2^{\ell}-5 }{ \ell}$ possibilities for the set of chips that can end up in the zigzag of the stable configuration of a tree. Similar to before, the number of possible subsequences of chips in the zigzag is $A_{\ell}\binom{2^{\ell}-5 }{\ell}$. Now we need to partition $2^{\ell}-5$ into the subtree. As the largest subtree contains two chips in two fixed places, we need to add to it $2^{\ell-1}-3$. Similarly, the second largest subtree needs an extra $2^{\ell-3}-3$. Overall, the number of ways to distribute the remaining chips into subtrees is $\binom{2^{\ell}-5-\ell }{ 2^{\ell-1}-3, 2^{\ell-2}-3, 2^{\ell-3}-1, 2^{\ell-4}-1, \dots 1}$. Continuing as above, we get our final estimate.
\end{proof}

\begin{example}\label{ex:T4Computation}
    The value of $T_4$ is bounded by
    \[10 A_{4}\binom{2^{4}- 3}{4} \binom{2^{4}-3-4 }{ 2^{4-1}-2, 2^{4-2}-2, 2^{4-3}-1} = 10 \cdot 5 \cdot \frac{13!}{4!9!}\cdot \binom{9 }{ 6, 2, 1} = 55 \cdot 715 \cdot 252 = 9,009,000.\]
    The value of $Z_4$ is bounded by
    \[10 A_{4}\binom{2^{4}- 5}{4} \binom{2^{4}-5-4 }{ 2^{4-1}-3, 2^{4-2}-3, 2^{4-3}-1} = 10 \cdot 5 \cdot \frac{11!}{4!7!}\cdot \binom{7 }{ 5, 1, 1} = 55 \cdot 330 \cdot 42 = 693,000.\]
    This is better than the trivial bound for $T_4$, which is $13!=6,227,020,800$, or the trivial bound for $Z_4$, which is $11! = 39,916,800$.
\end{example}

The explicit formula is shown in Theorem~\ref{thm:explicit}. For our purposes let $\beta_{\ell}$ be a shorthand for an expression
\[\beta_{\ell} = A_{\ell} \cdot \binom{2^{\ell}- 3}{\ell} \binom{2^{\ell}-3 -\ell }{2^{\ell-1}-2, 2^{\ell-2}-2, 2^{\ell-3}-1, 2^{\ell-4}-1, \dots, 1}\]
and $\gamma_{\ell}$ be a shorthand for an expression
\[\gamma_{\ell} = A_{\ell} \cdot \binom{2^{\ell}- 5}{\ell} \binom{2^{\ell}-5 -\ell }{2^{\ell-1}-3, 2^{\ell-2}-3, 2^{\ell-3}-1, 2^{\ell-4}-1, \dots, 1}.\]

\begin{theorem}[Zigzag Bound]
\label{thm:explicit}
    If we start with $2^{\ell} - 1$ labeled labeled chips at the root for $\ell \geq 4$, then 
    \[T_\ell \leq 10^{2^{\ell-4}} \beta_{\ell} \prod_{i=4}^{\ell-1} \beta_{i} ^{2^{\ell-1-i}}\]
    and
    \[Z_\ell \leq 10^{2^{\ell-4}} \gamma_{\ell} \prod_{i=4}^{\ell-1} \beta_{i} ^{2^{\ell-1-i}}.\]
\end{theorem}
    \begin{proof}
    We prove the first inequality by induction. Recall from Example~\ref{ex:T4Computation} that
    \[T_4 \leq 10^{2^{4-4}} \beta_{4} \prod_{i=4}^{4-1}\ \beta_{i}^{2^{4-1-i}} = 10 \beta_4 = 9,009,000.\]

    Assume that for integer $k \geq 4$ and $ k' \in\{4, 5, 6,\dots, k\}$, the value $T_{k'}$ is at most
    \[10^{2^{k'-4}} \beta_{k'}\prod_{i=4}^{k'-1} \beta_i ^{2^{k'-1-i}}.\]

Plugging this into the first equality in Lemma~\ref{lem:recursiveupper}, we get
\begin{multline*}
    T_{k+1} \leq 10 \beta_{k+1}\prod_{k'=4}^k \left(10^{2^{k'-4}}\beta_{k'}\prod_{i=4}^{k'-1} \beta_i^{2^{k'-1-i}}\right) \\
    = 10^{1+\sum_{k'=4}^k2^{k'-4}}\beta_{k+1} \prod_{k'=4}^k \beta_{k'}^{1+(\sum_{j=0}^{k-1-k'}2^j)} = 10^{2^{k-3}} \beta_{k+1} \prod_{k'=4}^k \beta_{k'}^{2^{k-k'}},
\end{multline*}
hereby proving the inductive step.

The second inequality is proven similarly.

\end{proof}

The values for this bound for $4 \leq \ell \leq 7$ are in Table~\ref{tab:CalculationsOfBounds}.

\subsection{Another potential bound}

In this section, we explore a different approach for calculating the bound.

\begin{definition}
We say that chip distribution on a binary tree satisfies the \textit{ballot property} if for each $i \in \mathbb{N}$ the $i$th smallest chip in the subtree rooted at the left child of the subtree's root has a smaller label than the $i$th smallest chip in the subtree rooted at the right child.
\end{definition}

In the next example, we show that a subtree with two or three layers that are a part of the stable configuration satisfies the ballot property. 

\begin{example}
   Suppose we have a tree with two layers: a root and two children. By Proposition~\ref{prop:MNfinalconfig}, the left child contains the smallest chip, and the right child contains the largest chip. The ballot property follows. Suppose we have a subtree with $7$ vertices occupying three layers. Let us renumber the chips corresponding to their order. We know that chip $1$ went to the left child and chip $7$ to the right child. Suppose chip $2$ is on the left branch. Then the left subtree has chips $(1,2,a)$, and the right subtree has chips $(b,c,7)$. Thus, the ballot property is satisfied. Suppose chip $2$ is on the right branch. Then the left tree has chips $(1,a,b)$, and the right tree has chips $(2,c,7)$, where $a$ and $c$ are, respectively, the chips at the left and right children of the root $r$. Letting $r$ be the chip at root, we know that $a <r$ and $r < c$, implying that $a < c$. In addition, $b < 7$. Thus, the ballot property is satisfied, too.
\end{example}

We also wrote a program \cite{GithubCode} to check the data available in \cite{musiker2023labeledchipfiringbinarytrees} that the ballot property is satisfied for every stable configuration when we start with $15$ chips.

\begin{conjecture}
\label{conj:BallotProperty}
    In a stable configuration, the whole tree and any subtree satisfy the ballot property.
\end{conjecture}

Consider the number of subtrees in stable configurations that satisfy the ballot property and have $2^i-1$ chips. Our Conjecture~\ref{conj:BallotProperty} is equivalent to saying the number of such trees is $T_i$.

We also use notation $C_n = \frac{1}{n+1}\binom{2n }{n}$, for the $n$th Catalan number. 

We begin with a lemma bounding the number of possible values that can end up at the root of a tree/subtree in the stable configuration.

\begin{lemma}
\label{lem:chipsatroot}
    If $\ell \geq 4$, then the number of possible values of chips that end up at the root of the tree in the stable configuration is at most $2^\ell - 7$. If a subtree in the stable configuration has $2^\ell -1$ chips, then the number of possible values of chips that end up at the root of the subtree is at most $2^\ell - 4$ chips.
\end{lemma}

\begin{proof}
    We know from Lemma~\ref{lem:SecondLargestAndSmallest} and Proposition \ref{prop:MNbottomsstraightleft} that the two smallest chips cannot end up at the root. Because of this and Proposition \ref{prop:MNfinalconfig}, the chip at the left child of the root is smaller than the root's chip and larger than the second smallest chip, given that $\ell \geq 4$. Thus, there are at least three chips that are smaller than the root's chip and not on the root. Hence, the chip at the root in the stable configuration has a label greater than $3$. By symmetry, it is less than $2^\ell-3$, leaving for $2^\ell - 7$ possibilities.

    We now discuss a subtree. We know that the smallest and largest chips cannot end up at the root. Suppose the root is the left child of its parent. Then, the chip at the right child of the root is greater than the chip at the root. Thus, the root cannot have the second-greatest chip. Similarly, if the root of the subtree is the right child, then the chip at the root cannot have the second-smallest chip. Thus, we can exclude three chips from being able to be at the root of the subtree.
\end{proof}

Using the lemma, we now provide a recursive relation for the bound on $T_{\ell}, Z_{\ell}.$

\begin{lemma}\label{lem:NewGoodArgument}
    Suppose that Conjecture~\ref{conj:BallotProperty} holds, then
    \[T_{\ell} \leq (2^\ell -4) C_{2^{\ell-1} - 1} T_{\ell - 1}^2 \quad \textrm{ and } \quad Z_{\ell} \leq (2^\ell -7) \left(\binom{2^{\ell}-6 }{ 2^{\ell-1}-3} - \binom{2^{\ell}-6 }{ 2^{\ell-1}-6}\right)T_{\ell - 1}^2.\]
\end{lemma}
\begin{proof}
    From Lemma~\ref{lem:chipsatroot}, we know that in the stable configuration, there are at most $2^\ell - 4$ chips that can end up at the root of a subtree, as long as $\ell \geq 4$ and at most $2^\ell - 7$ chips that can end up at the root of the tree.

    We start with a proper subtree. The $2^{\ell}-2$ chips that are below the root $v$ in our stable configuration are split evenly between the left and right branches of the root. The ballot property means that the number of ways the chips can be split between the branches is a Catalan number. Thus, there are at most $ C_{2^{\ell-1}-1}$ ways the chips can be split between the branches.

    For the original tree with self-loop, Lemma \ref{lem:SecondLargestAndSmallest} guarantees that chips $2$ and $2^{\ell}-2$ will be in the left and right subtrees, respectively. Thus, the number of ways to distribute the $2^{\ell}-1$ chips to the left and right subtrees is, at most, the number of ballot sequences of length $2^{\ell}-1$ where neither the beginning nor the end has two different votes. The number of such ballot sequences is described by OEIS entry A026012 \cite{oeis} and is equal to
    \[\binom{2^{\ell}-6}{ 2^{\ell-1}-3} - \binom{2^{\ell}-6 }{ 2^{\ell-1}-6}.\]

    Finally, observe that there are $2^{\ell-1}-1$ distinct, labeled chips in each subtree rooted by a child of $v$ and that there are at most $T_{\ell-1}$ ways in which these chips can appear in the subtree. Since $v$ has two children, we obtain that there are $T_{\ell-1}^2$ final placements of chips in the left and right subtrees.

    Thus, there are not more than $(2^{\ell}-4) C_{2^{\ell-1}-1} T_{\ell-1}^2$ stable configurations of the subtree with $2^{\ell}-1$ chips and not more than $(2^{\ell}-7) \left(\binom{2^{\ell}-6}{ 2^{\ell-1}-3} - \binom{2^{\ell}-6 }{ 2^{\ell-1}-6} \right) T_{\ell-1}^2$ if our tree is not a subtree, but the original tree.
\end{proof}

We produce a direct formula for the bounds on $T_{\ell}$ and $Z_{\ell}$.

\begin{theorem}[Ballot Bound]
    Suppose that Conjecture~\ref{conj:BallotProperty} holds. Then, for any integer $\ell \geq 3$, we have 
    \[T_{\ell} \leq 10^{2^{\ell-3}} \prod_{i=0}^{\ell-4}((2^{\ell-i}-4)C_{2^{\ell-i-1}-1})^{2^i}\] and
\[Z_{\ell} \leq (2^{\ell}-7) \cdot \left(\binom{2^{\ell}-6 }{ 2^{\ell-1}-3} - \binom{2^{\ell}-6 }{ 2^{\ell-1} -6} \right) \cdot 10^{2^{\ell-3}} \prod_{i=0}^{\ell-5}((2^{\ell-1-i}-4)C_{2^{\ell-i-2}-1})^{2^{i+1}}.\]
\end{theorem}

\begin{proof}
We prove the first inequality inductively. Recall that $T_{3} = 10$. Our formula give us $T_3 \leq 10^{2^{3-3}}\prod_{i=0}^{3-4}((2^{\ell-i}-4) C_{2^{\ell-i-1}-1})^{2^i} = 10$. Thus, the base of induction holds. Then, for our inductive hypothesis, assume that for $\ell \geq 3$, we have
\[T_{\ell} \leq 10^{2^{k-3}}\prod_{i=0}^{\ell-4}((2^{\ell-i}-4)C_{2^{\ell-i-1}-1})^{2^i}.\]
We will prove that $T_{\ell+1} \leq \prod_{i=0}^{\ell-3}((2^{\ell+1-i}-4)C_{2^{\ell-i}-1})^{2^i}.$ To see this observe that
\begin{equation*}
\begin{split}
T_{\ell+1} \leq (2^{\ell+1}-4)C_{2^{\ell}-1} T_{\ell}^2 \leq (2^{\ell+1}-4)C_{2^{\ell}-1}\left(10^{2^{\ell-3}}\prod_{i=0}^{\ell-4}((2^{\ell-i}-4) C_{2^{\ell-i-1}-1})^{2^i}\right)^2 \\
= 10^{2^{\ell-3}}\prod_{i=0}^{\ell-3}((2^{\ell+1-i}-4)C_{2^{\ell-i}-1})^{2^i},
\end{split} 
\end{equation*}
where the inequality results from using Lemma~\ref{lem:NewGoodArgument} and the equality from applying the inductive hypothesis.

We now prove the second inequality. For $\ell =3$, the inequality holds. This is since $Z_3 = 6$ and since $(2^3-7) \cdot \left(\binom{2^{3}-6 }{ 2^{3-1}-3} - \binom{2^{3}-6}{ 2^{3-1}-6}\right) \cdot 10^{2^{3-3}} \prod_{i=0}^{-2}((2^{3-1-i}-4)C_{2^{\ell-i-2}-1})^{2^{i+1}} = 1 \cdot (\binom{2 }{1} - \binom{2 }{-2}) \cdot 10 = 20.$
For $\ell \geq 4$, Lemma \ref{lem:NewGoodArgument} implies $Z_{\ell} \leq (2^{\ell}-7) \cdot \left(\binom{2^{\ell}-6 }{ 2^{\ell-1}-3} - \binom{2^{\ell}-6 }{ 2^{\ell-1}-6}\right) T_{\ell-1}^2$, we obtain \begin{equation}\begin{split}
    Z_{\ell} & \leq (2^{\ell}-7) \cdot \left(\binom{2^{\ell}-6}{ 2^{\ell-1}-3} - \binom{2^{\ell}-6 }{2^{\ell-1}-6}\right) T_{\ell-1}^2 \\ &  \leq  (2^{\ell}-7) \cdot \left(\binom{2^{\ell}-6 }{ 2^{\ell-1}-3} - \binom{2^{\ell}-6}{ 2^{\ell-1}-6}\right) \cdot 10^{2^{\ell-3}} \cdot \left( \prod_{i=0}^{\ell-5} ((2^{\ell-1-i}-4) \cdot C_{2^{\ell-i-2}-1})^{2^{i+1}}\right)
\end{split}\end{equation} with the last inequality being the result of our bound on $T_{\ell}$ for $\ell \geq 3.$ \end{proof}

\begin{example}
    We get $Z_4 \leq (2^4-7)(\binom{2^{4}-6}{ 2^{4-1}-3} - \binom{2^{4}-6}{2^{4-1}-6})T_3^2 = 9 \cdot (\binom{10}{5}  - \binom{10}{2} ) \cdot 10^2 = 186300$.
\end{example}

\subsection{Zigzag and ballot bounds are better than the naive bound}

We summarize our bounds for $Z_\ell$ for small values of $\ell$ in Table~\ref{tab:CalculationsOfBounds}.

\begin{table}[H]
    \centering
    \begin{tabular}{|c|c|c|c|}
    \hline
        $\ell$ & Naive Bound & Zigzag Bound & Ballot Bound \\ \hline
         $4$ &  $39,916,800$  & $ 693,000$ & $186, 300$ \\ $5$ & $\approx 1.1 \cdot 10^{28}$ & $\approx 2.9 \cdot 10^{22}$ & $ \approx 3.4 \cdot 10^{19}$ \\ $6$ & $\approx 1.4 \cdot 10^{80}$ & $\approx 1.8 \cdot 10^{65}$ & $\approx 2.3 \cdot 10^{57}$ \\ $7$ & $\approx 1.2 \cdot 10^{205}$ & $\approx 1.5 \cdot 10^{170}$ & $\approx 1.3 \cdot 10^{152}$\\
         \hline 
     \end{tabular}
    \caption{Comparing bounds for small $\ell$}
    \label{tab:CalculationsOfBounds}
\end{table}

We see in the table that for small $\ell$, our bounds are better than the naive bound. We now prove this for any $\ell$.

We define sequences $T_{\zigzag}(\ell)$, $Z_{\zigzag}(\ell)$, $T_{\ballot}(\ell)$, and $Z_{\ballot}(\ell)$ to match the corresponding bounds:
\[T_{\zigzag}(\ell) = 10^{2^{\ell-4}}\beta_{\ell}\prod_{i=4}^{\ell-1}\beta_i^{2^{\ell-1-i}} \quad \textrm{ and } \quad Z_{\zigzag}(\ell) = 10^{2^{\ell-4}} \gamma_{\ell} \prod_{i=4}^{\ell-1}\beta_i^{2^{\ell-1-i}},\]
while
\[T_{\ballot}(\ell) = 10^{2^{\ell-3}}\prod_{i=0}^{\ell-4}((2^{\ell-i}-4)C_{2^{\ell-i-1}-1})^{2^i}\] 
and
\[Z_{\ballot}(\ell)= (2^{\ell}-7) \left(\binom{2^{\ell}-6}{2^{\ell-1}-3} - \binom{2^{\ell}-6}{2^{\ell-1}-6}\right)10^{2^{\ell-3}} \left(\prod_{i=0}^{\ell-5}((2^{\ell-1-i}-4)C_{2^{\ell-i-2}-1})^{2^{i+1}}\right).\]

We know that $Z_\ell < T_\ell$, for $\ell > 2$. The same is true for the bounds, as we prove in the following lemma.

\begin{lemma}
\label{lem:ZboundlessthenTbound}
    We have $Z_{\zigzag}(\ell) < T_{\zigzag}(\ell)$ and $Z_{\ballot}(\ell) < Z_{\ballot}(\ell)$ for $\ell \geq 4$.
\end{lemma}

\begin{proof}
    By definition 
    \[\frac{Z_{\zigzag}(\ell)}{T_{\zigzag}(\ell)} = \frac{\gamma_{\ell}}{\beta_{\ell}} = \frac{\binom{2^\ell-5}{\ell} \binom{2^{\ell}-5-\ell}{2^{\ell-1}-3,2^{\ell-2}-3, 2^{\ell-3}-1, 2^{\ell-4}-1, \dots, 1}}{\binom{2^\ell-3}{\ell} \binom{2^{\ell}-3-\ell}{2^{\ell-1}-2,2^{\ell-2}-2, 2^{\ell-3}-1, 2^{\ell-4}-1, \dots, 1}}  < 1.\]
    Similarly, by definition
    \[\frac{Z_{\ballot}(\ell)}{T_{\ballot}(\ell)} = \frac{(2^\ell-7)\left(\binom{2^{\ell}-6 }{ 2^{\ell-1}-3} - \binom{2^{\ell}-6 }{ 2^{\ell-1}-6}\right)} {(2^\ell -4) C_{2^{\ell-1} - 1}}< 1.\]
\end{proof}

This means that to show that the bounds $Z_{\zigzag}$ and $Z_{\ballot}$ are better than the naive bound $(2^{\ell}-7)!$, it is enough to show this for bounds $T_{\zigzag}$ and $T_{\ballot}$.

\begin{theorem}
For all positive integers $\ell \geq 5$, the  zigzag and ballot bounds for $Z_\ell$ and $T_\ell$ are both less than $(2^{\ell}-7)!$.
\end{theorem}
\begin{proof}
    We start with the zigzag bound. Observe that $T_{\zigzag}(5) = 100 \beta_{5}\beta_4$ has 23 digits, while $(2^5-7)!$ has 26 digits. This serves as the base of induction. Note that $A_{\ell}  < \ell!$ by definition, and we obtain that 
    \begin{equation*}
    \begin{split}
    T_{\zigzag}(\ell+1) = &  10 A_{\ell+1}\binom{2^{\ell+1}-3 }{ \ell+1} \binom{2^{\ell+1}-4-\ell}{ 2^{\ell}-2, 2^{\ell-1}-2, 2^{\ell-2}-1, 2^{\ell-3}-1, \dots, 1} \prod_{i=4}^{\ell-1}T_{\zigzag}(i)  \\ 
    < &  10 (\ell+1)! \cdot \frac{(2^{\ell+1}-3)!}{(\ell+1)! (2^{\ell+1}-4-\ell)!} \cdot \frac{(2^{\ell+1}-4-\ell)!}{(2^{\ell}-2)!(2^{\ell-1}-2)!\prod_{i=2}^{\ell-1}(2^{\ell-i}-1)!}\prod_{i=0}^{\ell-4}(2^{\ell-i} -7)!\\
    = & 10 \cdot (2^{\ell+1}-3)! \cdot \frac{(2^\ell-7)!(2^{\ell-1}-7)!}{(2^\ell-2)!(2^{\ell-1}-2)!} \cdot \frac{1}{(2^1-1)!(2^2-1)!(2^3-1)!} \prod_{i=2}^{\ell-4} \frac{(2^{\ell-i} -7)!}{(2^{\ell-i}-1)!} \\
    = & \frac{5}{12} \cdot (2^{\ell+1}-3)! \cdot \frac{(2^\ell-7)!(2^{\ell-1}-7)!}{(2^\ell-2)!(2^{\ell-1}-2)!} \cdot \prod_{i=2}^{\ell-4} \frac{(2^{\ell-i} -7)!}{(2^{\ell-i}-1)!} \\
    < & (2^{\ell+1}-3)! \cdot \frac{(2^\ell-7)!(2^{\ell-1}-7)!}{(2^\ell-2)!(2^{\ell-1}-2)!} \\
    < & (2^{\ell+1}-3)! \cdot \frac{1}{(2^{\ell}-6)^5(2^{\ell-1}-6)^5} <(2^{\ell+1}-3)! \cdot \frac{1}{(2^{\ell}-6)^{8}} .
    \end{split}
    \end{equation*}
    What is left to show is that
    \[\frac{(2^{\ell+1}-3)!}{(2^{\ell+1}-7)!} = (2^{\ell+1}-3)(2^{\ell+1}-4)(2^{\ell+1}-5)(2^{\ell+1}-6) < (2^{\ell}-6)^8.\]
    The left side is less than $2^{4\ell+4}$. Thus, it is enough to show that $2^{\ell+1} < (2^{\ell}-6)^2$, which is true for $\ell \geq 5$.

   We now prove inductively that $T_{\ballot}(\ell) < (2^{\ell}-7)!$ for $\ell \geq 5$. Observe that the $T_{\ballot}(5) = (C_{15} \cdot (2^5-4)) \cdot (10^2 \cdot (2^4-4) \cdot C_7)^2 = 71940918415766400000 < (2^5-7)!$. For our inductive step, we show that $T_{\ballot}(\ell+1) < (2^{\ell+1}-7)!$, while we assume that $T_{\ballot}(\ell) < (2^{\ell}-7)!$. We know that $T_{\ballot}(\ell+1) = (2^{\ell+1}-4)C_{2^{\ell}-1} T_{\ballot}(\ell)^2 = (2^{\ell+1}-4) \frac{1}{2^{\ell}} \binom{2^{\ell+1}-2}{ 2^{\ell}-1} T_{\ballot}(\ell)^2$. We have
   
\begin{equation*}
    \begin{split}
        T_{\ballot}(\ell+1) < &  (2^{\ell+1}-4) \frac{(2^{\ell+1}-2)(2^{\ell+1}-3)(2^{\ell+1}-4) (2^{\ell+1}-5)(2^{\ell+1}-6) (2^{\ell+1}-7)!}{2^{\ell}((2^{\ell}-1)!)^2} ((2^{\ell}-7)!)^2 \\ & = (2^{\ell+1}-4)\cdot  \frac{1}{2^{\ell}} \cdot \frac{(2^{\ell+1}-2)(2^{\ell+1}-3)(2^{\ell+1}-4)(2^{\ell+1}-5)(2^{\ell+1}-6)}{(2^{\ell}-1)^2(2^{\ell}-2)^2(2^{\ell}-3)^2(2^{\ell}-4)^2(2^{\ell}-5)^2(2^{\ell}-6)^2} \cdot (2^{\ell+1}-7)! \\ &  = \frac{1}{2^{\ell}} \cdot \frac{16(2^{\ell+1}-3)(2^{\ell+1}-5)}{(2^{\ell}-1)(2^{\ell}-3)(2^{\ell}-4)^2(2^{\ell}-5)^2(2^{\ell}-6)^2} \cdot (2^{\ell+1}-7)! < (2^{\ell+1}-7)!,
    \end{split}
\end{equation*}
with the last inequality being a result of $2^{\ell-4}(2^{\ell}-1)(2^{\ell}-3)(2^{\ell}-4)^2(2^{\ell}-5)^2(2^{\ell}-6)^2 > (2^{\ell+1}-3)(2^{\ell+1}-5)$ when $\ell \geq 5$. 

Using Lemma~\ref{lem:ZboundlessthenTbound} and the fact that the zigzag and ballot bounds for $T_{\ell}$ are less than $(2^{\ell}-7)!$, the zigzag and ballot bounds for $Z_{\ell}$ are also less than $(2^{\ell}-7)!$.
\end{proof}

\section{Counting the number of fires}
\label{sec:countingvertexfires}

\subsection{The number of times each vertex fires}

In \cite{musiker2023labeledchipfiringbinarytrees}, Musiker and Nguyen count the number of times each vertex fires when starting with $2^\ell - 1$ unlabeled chips at the root. Therefore, a natural question arises: what happens if we start with $N$ chips at the root for any $N \in \mathbb{N}^+$? We generalize counting the number of times each vertex fires when starting with any $N$ unlabeled chips at the root where $N$ is a positive integer. 

We use Proposition~\ref{prop:numchipseachlayer}, which calculates the number of vertices for each layer in the stable configuration if we start with $N$ chips. Similar to Proposition~\ref{prop:numchipseachlayer}, we denote by $a_na_{n-1}\ldots a_2a_1a_0$ the binary expansion of $N+1$. 

Let us denote by $c_{i}(N)$ for $0 \leq i \leq n-1$ the number of chips in the stable configuration at each vertex on layer $i+1$ when we start with $N$ chips at the root. We have $c_{i}(N) = a_i(N)+1$. We also denote by $f_i(N)$ the total number of firings by each vertex on layer $i+1$. We sometimes use the notation $f_i$ and $c_i$ when it is clear what $N$ is. Notice that $a_n = 1$ and is ignored in this discussion, as the last layer corresponds to $i=n-1$.

\begin{lemma}
\label{lem:NumberFirings}
    The difference in the number of firings $f_i(N)-f_{i+1}(N)$ equals the number of chips in the stable configuration belonging to all children of a vertex on layer $i+1$ divided by $2$. In other words,
    \[f_i(N)-f_{i+1}(N)=\sum_{j=i+1}^{n-1}  2^{j-i-1} c_{j}(N),\]
    where $n = \lfloor \log_2(N+1) \rfloor$.
\end{lemma}

\begin{proof}
    The number of firings of a vertex equals the number of chips that are sent down divided by 2. The number of chips sent down equals the number of chips below the given vertex in the stable configuration plus the number of chips sent back.
    
    The number of vertices on layer $j+1$ is $2^j$, each containing $c_j(N)$ chips by Corollary 3.4 of \cite{musiker2023labeledchipfiringbinarytrees}. There are $\sum_{j=i+1}^{n=1}2^{j-i}c_j(N)$ below the given vertex, which is in the $i+1$st layer.

    The number of chips sent back is $2f_{i+1}$, as each child of our vertex fires $f_{i+1}$ times sending one chip up. The result follows.
\end{proof}

\begin{example}
    As the last layer never fires, we get $f_{n-1} = 0$. From here, we get $f_{n-2} = c_n$ since each vertex gets a chip exactly when its parent fires.
\end{example}

\begin{theorem}
\label{thm:vertexfires}
Given the total number of chips $N$, and the index $n = \lfloor \log_2(N+1) \rfloor$, the number of fires for each vertex on layer $n-i+1$ is the number of chips below it in the stable configuration minus the number of chips on one branch below it. In other words,
\[f_{n-i} = \sum_{j=1}^{i-1} (2^j-1)c_{n-i+j}.\]
\end{theorem}

\begin{proof}
We prove this by induction, starting with $i=1$. As the last layer never fires, we have $f_{n-1} = 0$. Our formula gives us the same $f_{n-1} = \sum_{j=1}^0 (2^j-1)c_{n-i+j} = 0$. Thus, the base of induction is established.

Suppose the formula is true for $i\in \{1, \ldots, n-1\}$. We want to calculate the number of firings for $i+1$, in other words we want to find $f_{n-i-1}$ assuming that $f_{n-i} = \sum_{j=1}^{i-1} (2^j-1)c_{n-i+j}.$

From Lemma~\ref{lem:NumberFirings} we know that
\[f_{n-i-1} - f_{n-i} = \sum_{j=n-i}^{n-1} 2^{j-n+i} c_{j}.\]
Thus, after doing some arithmetic and keeping track of indices, we get
\begin{equation*}
\begin{split}
    f_{n-i-1} = f_{n-i} + \sum_{j=n-i}^{n-1}   2^{j-n+i} c_{j} 
    = \sum_{j=1}^{i-1} (2^j-1)c_{n-i+j} + \sum_{j=n-i}^{n-1}  2^{j-n+i} c_{j} \\=\sum_{j=1}^{i-1} (2^j-1)c_{n-i+j} + \sum_{j=0}^{i-1}  2^{j} c_{n-i+j} =
    c_{n-i} + \sum_{j=1}^{i-1}(2^{j+1}-1)c_{n-i+j}
    \\ = \sum_{j=1}^{i}(2^j-1)c_{n-i+j-1}.
\end{split}
\end{equation*}
This concludes the proof.
\end{proof}

\begin{example}\label{ex:TheMusikerExample}
    If we start with $2^n-1$ chips at the root, then $c_i = 1$ for $0 \leq i < n$. The number of chips in the stable configuration below a vertex on layer $n-i$ is $2^i - 1$. The number of chips on one branch is $i$. Thus, the total number of fires of such vertex is $2^i - i - 1$. This is the same answer as in \cite{musiker2023labeledchipfiringbinarytrees}.
\end{example}

\subsection{The number of root fires}

As a particular case of Theorem~\ref{thm:vertexfires}, we get the number of root fires.

\begin{corollary}\label{cor:TotalFromRoot}
    Given the total number of chips $N$, and the index $n = \lfloor \log_2(N+1) \rfloor$, the number of times the root fires is
    \[f_0 = \sum_{j=1}^{n-1} (2^{j}-1)c_{j}.\]
\end{corollary}

For completeness, we also provide a recursive formula for $f_0$.

\begin{corollary}\label{cor:FiringRecursive}
    Given the total number of chips $N$ and the index $n = \lfloor \log_2(N+1) \rfloor$, the recursive formula $f_0(N) = \lceil N/2\rceil - 1 + f_0(\lceil N/2\rceil - 1)$ holds, where $f_0(1) = 0$.
\end{corollary}

\begin{proof}
We know that $f_0(1) = 0$. Now for the recursion.

Given that $N+1$ is $a_na_{n-1}\ldots a_2a_1a_0$ in binary, it follows that 
$\left\lceil \frac{N}{2} \right\rceil$ in binary is $a_na_{n-1}\ldots a_1$. Recall that we denoted $a_i+1$ as $c_i(N)$ or $c_i$. We have
\begin{equation*}
\begin{split}
\left\lceil \frac{N}{2} \right\rceil - 1 = -1 + \sum_{j=1}^na_j2^{j-1} = a_n 2^{n-1} - 1 + \sum_{j=1}^{n-1}(c_{j} - 1)2^{j-1} = \\
= 2^{n-1} - 1 + \sum_{j=1}^{n-1}c_{j}2^{j-1} - \sum_{j=1}^{n-1}2^{j-1} = \sum_{j=1}^{n-1}2^{j-1}c_{j}.
\end{split}
\end{equation*}

Now compute $f_0(N) - \left(\left\lceil \frac{N}{2}\right\rceil - 1\right)$. We get 
\[\sum_{j=1}^{n-1}(2^{j}-1)c_j - \sum_{j=1}^{n-1}2^{j-1}c_{j} = \sum_{j=1}^{n-1}(2^{j}-1 - 2^{j-1})c_{j} = \sum_{j=1}^{n-1}(2^{j-1}-1)c_{j}.\] 
After re-indexing the sum, we get
\[\sum_{j=0}^{n-2}(2^{j}-1)c_{j+1},\] 
which by definition is $f_0(\lceil N/2 \rceil - 1)$.
\end{proof}

We see that $f_0(N)$ does not depend on $a_0$ -- the last digit in the binary representation of $N+1$. Therefore, $f_0(2m-1) = f_0(2m)$. The sequence $f_0(2m)$ starts from index 1 as:
\[0,\ 1,\ 2,\ 4,\ 5,\ 7,\ 8,\ 11,\ 12,\ 14,\ 15,\ 18,\ 19,\ 21,\ 22,\ 26,\ \ldots.\]

The above sequence is the new sequence A376116 of OEIS \cite{oeis}, and given that it depends on the binary representation of the argument, it is interesting to look at the different sequence $f_0(2m+2) - f_0(2m)$.

\begin{proposition}\label{prop:DiffSeqOff0}
    If $m= 2^i-1$, then $f_0(2m+2) - f_0(2m) = i$, otherwise, if the binary expansion of $m$ ends with $i$ ones then $f_0(2m+2) - f_0(2m) = i+1$.
\end{proposition}

\begin{proof}
    If $m= 2^i-1$, then $f_0(2m+2) = f_0(2m+1) = f_0(2^{i+1} - 1) = 2^{i+1} - i - 2$, by Example~\ref{ex:TheMusikerExample}. On the other hand, when calculating $f_0(2m)$, we see that all the corresponding values of $c_j(2m) =2$. Thus, $f_0(2m) = 2\sum_{j=1}^{i-1} (2^{j}-1) = 2\sum_{j=1}^{i-1}2^{j}- 2\sum_{j=1}^{i-1}1 = 2(2^{i}-2)-2i+2 = 2^{i+1}-2i-2$. This means
    \[f_0(2m+2) - f_0(2m) = 2^{i+1}-i-2 - (2^{i+1}-2i-2) = i.\]

    Suppose $m$ cannot be represented as $2^n-1$ for any $n$. Also, suppose the binary expansion of $m$ ends with 0 followed by $i$ consecutive ones. This means the binary expansion of $2m+1$ ends with exactly $i+1$ consecutive ones. Since $f_0(2m+2)= f_0(2m+1)$, we have $f_0(2m+2)- f_0(2m) = f_0(2m+1) - f_0(2m)$, and we will compute $f_0(2m+1)-f_0(2m)$. We observe that for $j > i+1$, we have $a_j(2m+1) = a_j(2m)$ and hence $c_j(2m+1) = c_j(2m)$. Observe also that $a_{i+1}(2m) =0$ while $a_{i+1}(2m+1) = 1$. This means $c_{i+1}(2m)= 1$ and $c_{i+1}(2m+1)=2$. Finally, observe that for $j \in \{0, 1, \ldots, i\}$ we have $a_j(2m)= 1$ and $a_j(2m+1) = 0$, implying that $c_j(2m)= 2$ and $c_j(2m+1) = 1$. Thus, using the formula for $f_0$ from Corollary~\ref{cor:TotalFromRoot}, we compute
    \[f_0(2m+1)-f_0(2m) = (2^{i+1}-1) - \sum_{j=1}^{i}(2^j-1) = 2^{i+1}-1 - (2^{i+1}-2) + i = i+1,\]
    concluding the proof.
\end{proof}
It follows that the difference sequence of A376116 is the new sequence A091090 in the OEIS \cite{oeis}. It starts from index $2$ as 
\[1,\ 1,\ 2,\ 1,\ 2,\ 1,\ 3,\ 1,\ 2,\ 1,\ 3,\ 1,\ 2,\ 1,\ 4,\ 1,\ 2,\ 1,\ \ldots.\]

It can also be defined as the number of editing steps (delete, insert, or substitute) to transform the binary representation of $n$ into the binary representation of $n + 1$.

\subsection{The total number of fires}

Let us denote the total number of firings as $F(N)$. As we know the number of firings at each vertex, we can sum them up to calculate $F(N)$.

\begin{theorem}\label{thm:TotalNumberOfFirings}
    The total number of firings is 
    \[F(N) = \sum_{k=1}^{n-1}((k-1)2^{k}+1)c_k(N).\]
\end{theorem}

\begin{proof}
    Summing up the number of firings for all the vertices, we get
    \[F(N) = \sum_{i=0}^{n-2}2^if_i(N).\]
    After substituting $f_i(N)$ from Theorem~\ref{thm:vertexfires}, we get
    \[F(N) = \sum_{i=0}^{n-2}\sum_{j=1}^{n-i-1}(2^{i+j} - 2^i)c_{i+j}(N).\]
    After rearranging the addends and replacing $i+j$ with $k$, we have
    \[F(N) = \sum_{i=0}^{n-2}\sum_{j=1}^{n-i-1} 2^{i+j}c_{i+j}(N) - \sum_{i=0}^{n-2}\sum_{j=1}^{n-i-1}2^ic_{i+j}(N)= \sum_{k=1}^{n-1} 2^k k \cdot c_k(N) - \sum_{k=1}^{n-1}(1+2+4+ \cdots +2^{k-1})c_{k}(N).\]
    Combining the coefficients together, we get
    \[F(N)  =\sum_{k=1}^{n-1}2^k k  c_k(N) - \sum_{k=1}^{n-1}(2^k - 1)c_k(N) = \sum_{k=1}^{n-1}((k-1)2^{k}+1)c_k.\]
\end{proof}

Before giving an example, we present the following useful identity, which can be easily proven via induction:
\begin{equation}\label{eq:Usefulk2^kSum}
    \sum_{j=1}^{n}j2^j = (n-1)2^{n+1}+2, \quad \forall n \in \mathbb{Z}_{\geq 0}.
\end{equation}

\begin{example}\label{ex:TheMusikerExampleTotal}
    If we start with $2^n-1$ chips at the root, then $c_i = 1$ for $0 \leq i < n$. Thus, the total number of firings is $\sum_{k=1}^{n-1}((k-1)2^{k}+1) = \sum_{k=1}^{n-1}k2^{k} -\sum_{k=1}^{n-1}2^{k} +(n-1)$. Using the formula from Eq.~\ref{eq:Usefulk2^kSum}, we get
    \[F(2^n-1) = (n-2)2^n  + 2 + 2 - 2^n + (n-1) = (n-3)2^n + n +3,\]
     which is the same formula as in \cite[Corollary 3.7]{musiker2023labeledchipfiringbinarytrees}.
\end{example}

Similarly to the above, the total number of firings does not depend on the last digit in the binary expansion of $N+1$. Thus, $F(2m-1) = F(2m)$. The sequence $F(2m)$ starts from index 1 as
\[0,\ 1,\ 2,\ 6,\ 7,\ 11,\ 12,\ 23,\ 24,\ 28,\ 29,\ 40,\ 41,\ 45,\ 46,\ 72,\ 73,\ 77,\ 78,\ 89,\ 90,\ 94,\ 95,\ \ldots.\]

This is now the new sequence A376131. As before, we look at the difference sequence $F(2m+2) - F(2m)$, which starts as 
\[1,\ 1,\ 4,\ 1,\ 4,\ 1,\ 11,\ 1,\ 4,\ 1,\ 11,\ 1,\ 4,\ 1,\ 26,\ \ldots.\]

This is now sequence A376132. One might notice that this sequence contains Eulerian numbers $A(n+3,n) = 2^n-n-1$, which is sequence A000295 in the OEIS \cite{oeis}.
\begin{proposition}
    The difference sequence $F(2m+2) - F(2m)$ can be calculated as follows
    \[F(2m+2) - F(2m) = A000295(f_0(2m+2) - f_0(2m) + 1).\]
\end{proposition}

\begin{proof}
    Recall that $F(2n) = F(2n-1)$ and $f_0(2n) - f_0(2n-1)$ for any positive integer $n$. It suffices to show that $F(2m+2) - F(2m)$ equals $a(f_0(2m+2) - f_0(2m)+1)$.
    
    In the case where $m = 2^i - 1$ for some $i$, we obtain that $2m+1= 2^{i+1}-1$. Recall from \cite[Corollary 3.7]{musiker2023labeledchipfiringbinarytrees} that $F(2m+1) = 2^{i+1}(i-2) + (i+4) = i 2^{i+1} - 2^{i+2}  + i+ 4.$ On the other hand, observe that the binary expansion of $2m+1$ consists of $i+1$ ones. This tells us that for $j \in \{0, 1, 2, \ldots, i\}$, $c_j(2m) = 2$. Hence by Theorem~\ref{thm:TotalNumberOfFirings} and Eq.~\ref{eq:Usefulk2^kSum}, we obtain 
    \[
    \begin{split}F(2m) = \sum_{k=1}^{i-1}2((k-1)2^k  + 1) = 2 \sum_{k=1}^{i-1} k 2^{k} -  \sum_{k=1}^{i-1}2^{k+1} + 2i-2= \\
    2(2+2^{i}(i-2)) - (2^{i+1}-4) +2i-2 = i 2^{i+1}- 3 \cdot 2^{i+1}+2i+6.
    \end{split}\]
    Hence 
    \[
    \begin{split}
    F(2m+1)-F(2m) = (i2^{i+1}-2^{i+2} + i + 4)- (i 2^{i+1}- 3 \cdot 2^{i+1}+2i+6) \\
    = 2^{i+1} - i - 2 = A(i+4, i+1) = A000295(i+1).
    \end{split}\]
    Recalling that when $m= 2^i - 1$, Proposition~\ref{prop:DiffSeqOff0} tells us that $f_0(2m+1) - f_0(2m) = i$. This implies $F(2m+1) - F(2m) = A000295(i+1) = A000295(f_0(2m+1) - f_0(2m)+1)$ in this case.

    Now suppose that $m \neq 2^n -1$ for any $n$. Here, there exists some nonnegative integer $i$ such that $m$ with zero followed by $i$ consecutive ones. As observed in the second paragraph of the proof of Proposition~\ref{prop:DiffSeqOff0}, we have that $c_j(2m) = c_j(2m+1)$ for $j > i+1$, also $c_{i+1}(2m+1)=2 = c_{i+1}(2m)+1$, and $c_j(2m)= c_j(2m+1) + 1 = 2$ for $j \in \{0, 1, \ldots, i\}$. We therefore obtain using Eq.~\ref{eq:Usefulk2^kSum}
    \begin{equation*}
        \begin{split}
            F(2m+1) - F(2m) = (i 2^{i+1} + 1) - \sum_{k=1}^i ((k-1)2^k + 1) = (i 2^{i+1}+1) - \sum_{k=1}^ik 2^k + (2^{i+1} - 2)  - i 
            \\ = i2^{i+1} + 1 - ((i-1)2^{i+1} + 2) +2^{i+1} - 2 - i = 2^{i+2}-i-3 \\ = A(i+5, i+2) = A000295(i+2).
        \end{split}
    \end{equation*}
    Recall that $f_0(2m+2) = f_0(2m+1)$ and that Proposition~\ref{prop:DiffSeqOff0} tells us that since $m \neq 2^n-1$ and there are $f(2m+2) - f(2m) = i+1$, we obtain that, in this case, 
    \[F(2m+1) - F(2m) = A000295(i+2) = A000295(f_0(2m+1) - f_0(2m)+1).\]
\end{proof}

\section{Acknowledgments} 

This project started during the Research Science Institute (RSI) program. During RSI, many people helped, and we thank Professors David Jerison and Jonathan Bloom for overseeing the progress of the research problem. We thank Professor Alexander Postnikov for helpful discussions. Our appreciation goes to the RSI students and staff for creating a welcoming working environment.

The first and the second authors are financially supported by the MIT Department of Mathematics. The third author was sponsored by RBC Foundation USA.

\newcommand{\etalchar}[1]{$^{#1}$}

\end{document}